\UseRawInputEncoding
\documentclass{amsart}
\usepackage{amsmath}
\usepackage{amssymb}
\newtheorem{thm}{Theorem}[section]

\newtheorem{lem}[thm]{Lemma}
\newtheorem{prop}[thm]{Proposition}
\theoremstyle{definition}
\newtheorem{defn}[thm]{Definition}
\theoremstyle{remark}
\newtheorem{rem}[thm]{Remark}
\newtheorem{ex}[thm]{Example}

\def\Darr#1#2{{\scriptstyle #1}\downarrow{\scriptstyle #2}}
\def\RDarr#1#2{{\scriptstyle #1} \searrow {\!\!\!\! {}^{#2}}}
\long\def\Ref#1#2#3#4#5#6{
\bibitem{#1}
{\rm #2,}
\textit{#3.}
{\rm #4}
\textbf{#5}
{\rm #6.}
}
\long\def\Refb#1#2#3#4{
\bibitem{#1}
{\rm #2,}
\textit{#3.}
#4.
}
\def\Zz{{\mathbb Z}}
\def\Cc{{\mathbb C}}
\def\Ff{{\mathbb F}}
\def\into{\hookrightarrow}
\def\iso{\cong}
\def\st{\mid}
\def\Sq{{\rm Sq}}
\def\Sp{{\rm Sp}}
\def\Sing{\Sigma}
\def\phi{\varphi}
\def\leq{\leqslant}
\def\geq{\geqslant}
\begin{document}

\title{A Poincar\'e-Hopf theorem for n-valued \goodbreak\noindent
vector fields}
\author{M.~C.~Crabb}
\address{%
Institute of Mathematics,\\
University of Aberdeen, \\
Aberdeen AB24 3UE, UK}
\email{m.crabb@abdn.ac.uk}
\begin{abstract}
The Poincar\'e-Hopf theorem for line fields, as described in a paper
of Crowley and Grant \cite{CG}, is interpreted as a special case of a
Poincar\'e-Hopf theorem for $n$-valued sections of a vector bundle over
a closed manifold of the same dimension.
\end{abstract}
\date{February 2025}
\maketitle
\section{Introduction}
In \cite{CG} Crowley and Grant established a Poincar\'e-Hopf
theorem for line fields on a closed smooth $m$-manifold $M$.
A line field is a section, over some subspace of $M$,
of the real projective bundle $P(\tau M)$ of the tangent bundle $\tau M$.
More generally, one can consider sections of the projective bundle
$P(\xi )$ of any $m$-dimensional real vector bundle $\xi$ over $M$.
In this formulation, 
it is easy to deduce a Poincar\'e-Hopf theorem for projective
bundles by reduction to the classical case of sections of a sphere bundle.
This reduction, by methods that are essentially well known (but perhaps
not so familiar), is described in Section 2. 
It involves separate consideration of the cases in which $m$ is equal to
$2$ or is greater than $2$.
An analogous result for complex projective bundles is given in Section 3.

But the main result \cite[Theorem 1.1]{CG} of Crowley and Grant,
giving a uniform treatment of the cases $m=2$ and $m>2$, suggests
another generalization.
We assume that $\xi$ is equipped with a Euclidean inner product.
A section of a projective bundle $P(\xi )$ then determines, at each point $x$
of the base,
a $2$-element subset of the unit sphere $S(\xi_x)$ in the fibre $\xi_x$ of
$\xi$ at $x$,
namely the two points lying on the line in $\xi_x$ specified by the section.
Now, for any natural number $n\geq 1$, one could look at sections of
a suitable configuration space bundle determining at $x$ a subset of $S(\xi_x)$,
or just of $\xi_x-\{ 0\}$, with cardinality $n$.
Our main result, Theorem \ref{main}, is a direct extension of the classical
Poincar\'e-Hopf theorem for sections of a vector bundle to such $n$-valued
sections of $\xi$.
The definition of the local index of an $n$-valued section depends
on a degree theory for $n$-valued maps, which is described in 
Section 4. 
The methods derive from the work of Schirmer \cite{schirmer}
and Brown \cite{brown} on the fixed point theory of $n$-valued maps.

Throughout this paper,
$M$ will be a closed, connected, smooth manifold of dimension $m\geq 1$,
with orientation bundle denoted by $\delta$.
We recall the classical Poincar\'e-Hopf theorem
for an $m$-dimensional vector bundle  $\xi$ over $M$.
In the statement $e(\xi )$ is the Euler class of $\xi$ with 
$\Zz$-coefficients twisted by $\delta$. 
Note that, if $m$ is odd, then $2e(\xi )=0$ and so $e(\xi )[M]=0$.
\begin{thm} \label{poincare}
{\rm (Poincar\'e-Hopf).}
Let $\xi$ be an $m$-dimensional real vector bundle over $M$. Suppose
that $s$ is a section of $\xi$ which is non-zero outside a finite set $\Sing$.
\par\noindent {\rm (a).}
One can assign to each point of $\Sing$
a local index of $s$ in $\Zz /2\Zz$, and the sum of the
local indices is equal to $w_m(\xi )[M]\in\Zz /2\Zz$.
\par\noindent {\rm (b).}
Suppose that $w_1(\xi )=w_1(M)$ and an isomorphism is chosen between
$\delta$ and the orientation bundle of $\xi$.
Then one can assign to each point of $\Sing$ a local index of $s$ in $\Zz$, and the
sum of the local indices is equal to $e(\xi )[M]\in\Zz$.
\qed
\end{thm} 

The local index at a point $x\in \Sing$ is defined as follows. 
Fixing a Riemannian metric on $M$ so that we can talk about the closed
unit disc $D(\tau_xM)$ in the tangent space $\tau_xM$ at $x$,
we choose a small
tubular neighbourhood $D(\tau_xM) \into M$ of $x$, containing no other
point of $\Sing$, and a trivialization $\xi \, |\, D(\tau_xM) \iso D(\tau_x M) \times
\xi_x$ of $\xi$, restricting to the identity on $\xi_x$ at $x$,
over $D(\tau_xM)$.
Then $s$ determines, over the bounding sphere $S(\tau_xM)=\partial D(\tau_xM)$,
a map $\phi_x : S(\tau_x M) \to \xi_x-\{ 0\}$.
The homotopy class of $\phi_x$ is independent of the choices made in its
construction. We define the local index at $x$ to be
the degree of $\phi_x$ in case (b) using the correspondence between
the orientations of $\tau_xM$ and $\xi_x$, 
and the reduction of the degree (mod $2$) in case (a).
\section{Sections of real projective bundles}
In this section we explain how to reduce the study of sections
a real projective bundle to consideration of sections of a 
suitable sphere bundle.
We shall write $H$ for the (dual) Hopf line bundle over $P(\eta )$
constructed as a subbundle of (the pullback of) $\eta$.
\begin{prop}
Suppose that $\eta$ is an $m$-dimensional real vector bundle over
$M$ and that $\sigma$ is a section of the real projective bundle
$P(\eta )$ over the complement of a finite set $\Sing$.
\par\noindent {\rm (a).}
Suppose that $m>2$. Then $\sigma$ determines a class $l\in H^1(M;\,\Ff_2)$.
One can assign a local index of $\sigma$ in $\Zz /2\Zz$ to each point of $\Sing$, and
the sum of the local indices is equal to $(\sum_{i=0}^m l^iw_{m-i}(\eta ))[M]
\in\Zz /2\Zz$.
\par\noindent {\rm (b).}
Suppose that $m>2$ is even, that $w_1(\eta )=w_1(M)$ and that
an isomorphism  is chosen between $\delta$ and the orientation bundle
of $\eta$. Then one can assign to each point of $\Sing$ a local index  of $\sigma$ in $\Zz$,
and the sum of the local indices is equal to $e(\eta )[M]\in\Zz$.
\par\noindent {\rm (c).}
Suppose that $m=2$, that $w_1(\eta )=w_1(M)$ and that
an isomorphism  is chosen between $\delta$ and the orientation bundle
of $\eta$. Then one can assign to each point of $\Sing$ a local index of $\sigma$ in $\Zz$,
and the sum of the local indices is equal to $2e(\eta )[M]\in\Zz$.
\par\noindent {\rm (d).}
Suppose that $m=2$. Let $\Sing_+$ be the set of points $x\in \Sing$ such that
the pullback $\sigma^*H$ over $M-\Sing$ of the Hopf line
bundle $H$ over $P(\eta)$ extends over $(M-\Sing)\cup\{ x\}$, and write $\Sing_-=\Sing -\Sing_+$.
Then $\# \Sing_-$ is even.
If $\Sing_-=\emptyset$, so that $\sigma^*H$ extends
to a line bundle over $M$, classified by $l\in H^1(M;\,\Ff_2)$ say,
one can assign to each point of $\Sing$ a local index 
of $\sigma$ in $\Zz /2\Zz$
such that the sum of the local indices is equal to 
$(w_2(\eta)+lw_1(\eta )+l^2)[M]\in \Zz /2\Zz$.
\end{prop}
\begin{proof}
(a) and (b). 
These cases are considered in some detail in \cite{koschorke}
(but see, especially, \cite[Remark 1.9]{koschorke}).

Since $\dim M >2$,
the pullback $\sigma^*H$ over $M-\Sing$ of $H$ over $P(\eta )$
extends uniquely (up to homotopy) to a line bundle $\lambda$ over
$M$. The section $\sigma$ determines a tautological section $s$ of
$\lambda^*\otimes \eta$ given by the inclusion of $\lambda$ in $\eta$
(pulled back from the defining inclusion of $H$ in $\eta$).

We now apply Theorem \ref{poincare} to the section $s$ of 
$\xi = \lambda^*\otimes\eta$ over $M-\Sing$.
Let $l=w_1(\lambda )$.
Then we have $w_m(\lambda^*\otimes\eta ) =\sum l^iw_{m-i}(\eta )$
and
in case (b), using the natural isomorphism between the orientation
bundles of $\lambda^*\otimes\xi$ and $\xi$, we have
$e(\lambda^*\otimes\eta )[M]=e(\eta )[M]\in\Zz$,
because there is no $2$-torsion.

(c).  Let $\Cc_\delta$ denote the bundle of complex fields twisted by
$\delta$. (See, for example, \cite{CCV}.)
Then $\eta$ is a $\Cc_\delta$-line and we can identify
$P(\eta )$ with the sphere bundle $S(\eta\otimes_{\Cc_\delta}\eta )$
of the square of $\eta$. (This is just the twisted version of the standard
description for a complex line.)
We can now apply Theorem \ref{poincare} to $\xi =\eta\otimes_{\Cc_\delta }\eta$.
The Euler class of $\xi$ is $e(\xi ) = 2e(\eta )$
(calculated as the twisted first Chern class 
$c_1^\delta (\eta\otimes_{\Cc_\delta} \eta )=2 c_1^\delta (\eta )$).

(d). 
The first assertion can be read off from the exact sequence
$$
0\to H^1(M;\,\Ff_2) \to H^1(M-\Sing ;\,\Ff_2) \to \bigoplus_{x\in \Sing} \Ff_2 \to
\Ff_2=H^2(M;\,\Ff_2)\, .
$$
When $\Sing_-$ is empty, the class $l$ is uniquely determined by $\sigma$
and the result follows as in (a).
\end{proof}

\begin{rem}
In case (b), to be consistent with (a) it must be true that the sum
$\sum_{i=1}^ml^iw_{m-i}(\eta )$ is zero.
Indeed, we have
$l^{2j-1}w_{m-2j+1}(\eta )+l^{2j}w_{2j}(\eta )=0$, because 
$$
w_{m-2j+1}(\eta )= \Sq^1(w_{m-2j}(\eta ))+w_1(\eta )w_{m-2j}(\eta )
$$ 

\smallskip\par\noindent
and $w_1(M)l^{2j-1}w_{m-2j}(\eta )=\Sq^1(l^{2j-1}w_{m-2j}(\eta ))$,
by the formulae of Wu.
\end{rem}
\section{Sections of complex projective bundles}
Similar methods may be used for sections of the complex projective bundle 
$\Cc P(\eta )$ of a complex vector bundle $\eta$.
We shall again write $H$ for the dual Hopf line bundle over $\Cc P(\eta )$
embedded as the tautological subbundle of the pullback of $\eta$.
\begin{prop}\label{cx}
Suppose that $m=2k$ is even, $k>1$, that $M$ is oriented,
and that $\eta$ is a complex $k$-dimensional
vector bundle over $M$.
Suppose that $\sigma$ is a section of the complex projective bundle
$\Cc P(\eta )$ over the complement of a finite set $\Sing$.

Then $\sigma$ determines a class $l\in H^2(M;\,\Zz )$.
One can assign a local index of $\sigma$ in $\Zz$ to each point of $\Sing$, and
the sum of the local indices is equal to 
$(\sum_{i=0}^k (-l)^ic_{k-i}(\eta )[M]$.
\end{prop}
\begin{proof}
The pullback $\sigma^*H$ extends uniquely to a complex line bundle
$\lambda$ over $M$ with $c_1(\lambda )=l\in H^2(M;\,\Zz )$. The section
$\sigma$ determines a section $s$ of the sphere bundle 
of $\xi =\lambda^*\otimes_\Cc \eta$
to which we can apply Theorem \ref{poincare}.
The Euler class is $e(\xi )=c_k(\lambda^*\otimes_\Cc \eta ) =
\sum (-l)^ic_{k-i}(\eta )$.
\end{proof}
\begin{rem}
More generally, if $M$ is not oriented, one could consider
a $\Cc_\delta$-bundle $\eta$.
One has a $\Cc_\delta$-line bundle $\lambda$ classified by
$l\in H^2(M;\,\Zz (\delta ))$ and twisted Chern classes
$c_j^\delta (\eta )\in H^{2j}(M;\, \Zz (\delta^{\otimes j}))$.
(See, for example, \cite{CCV}.)
\end{rem}

If $l$ is a torsion class, then the sum of the local indices is equal
to $c_k(\eta )[M]=e(\eta )[M]$.
This leads on to consideration of bundles of lens spaces.
\begin{prop}
Suppose that $m=2k$ is even, $k\geq 1$, that $M$ is oriented,
and that $\eta$ is a complex $k$-dimensional vector bundle over $M$.
Fix $n\geq 1$ and consider the bundle of lens spaces
$S(\eta )/\mu_n \to M$, where $\mu_n\leq\Cc^\times$ is the group
of $n$th roots of unity acting on $\eta$ by multiplication.

Suppose that $\sigma$ is a section of $S(\eta )/\mu_n \to M$
over the complement of a finite set $\Sing$.
\par\noindent {\rm (a).}
Suppose that $k>1$. Then one can assign to each point of $\Sing$ a local index of $\sigma$ in
$\Zz$, and the sum of the local indices is equal to $e(\eta )[M]$.
\par\noindent {\rm (b).}
Suppose that $k=1$. Then one can assign to each point of $\Sing$ 
a local index of $\sigma$
in $\Zz$, and the sum of the local indices is equal to
$n\cdot e(\eta )[M]$.
\end{prop}
\begin{proof}
(a). This follows from Proposition \ref{cx}, because $nl=0$.

(b). We can identify $S(\eta )/\mu_n$ with the circle bundle
$S(\eta^{\otimes n})$.
The result then follows from Theorem \ref{poincare}, because
$e(\eta^{\otimes n})=n\cdot e(\eta )$.
\end{proof}
\section{The degree of an $n$-valued map}
Following the seminal work of Schirmer \cite{schirmer} and
Brown \cite{brown} on $n$-valued maps,
an $n$-valued map between two compact ENRs $X$ and $Y$ was described in
\cite{nval} by a finite cover $p : \tilde X \to X$ and a continuous map
$f: \tilde X\to Y$. 
We shall call such a pair $(f,p)$, denoted by the fraction $f/p$,
{\it a structured finite-valued map} from $X$ to $Y$.
It determines a multivalued map $F : X\multimap Y$ by $F(x)=f(p^{-1}(x))$.
We are concerned here with the case in which $X$ and $Y$ are smooth manifolds.
\begin{defn}
Given two closed, connected, oriented, smooth manifolds $X$ and $Y$ of the same 
dimension $m$, we define the {\it degree} of a structured 
finite-valued
map $f/p$ from $X$ to $Y$, where $p:\tilde X\to X$ is a finite cover and 
$f:\tilde X \to Y$ is a map, to be the integer $\deg (f/p)\in \Zz$
determined by the identity
$$
f_*[\tilde X] =\deg (f/p)\cdot [Y]
$$
involving the fundamental classes $[\tilde X]$ and $[Y]$ of the oriented
manifolds. (The orientation of $X$ lifts to an orientation of the
finite cover $\tilde X$.)
\end{defn}
\begin{rem}
If $Z$ is another closed, connected, oriented $m$-manifold
and a structured  finite-valued map $g/q$ from $Y$ to $Z$
is given by $q: \tilde Y\to Y$ and $g:\tilde Y\to Z$,
the composition $g/q\circ f/p$ is defined by the cover
$f^*\tilde Y \to \tilde X \to X$ and map $f^*\tilde Y \to \tilde Y \to Z$:
$$
\begin{matrix}
f^*\tilde Y  &&&& \\
\Darr{}{} &\!\!\!\!\! \RDarr{}{}\!\!\!\!\! & && \\
\tilde X&& \tilde Y && \\
\Darr{p}{\phantom{p}} &\!\!\!\!\!\RDarr{f}{\phantom{f}}\!\!\!\!\!& \Darr{q}{\phantom{q}}
 &\!\!\!\!\!\!\RDarr{g}{\phantom{g}}\!\!\!\!\!\!& \\
X && Y && Z
\end{matrix}
$$
The degree is multiplicative:
$\deg (g/q\circ f/p) =\deg (g/q)\cdot \deg (f/p)$.
\end{rem}
\begin{defn}
Let $V$ be a finite dimensional real vector space. 
For an integer $n\geq 1$, we write $C_n^*(V)$
for the configuration space of $n$-element sets of non-zero vectors in $V$.
Its canonical $n$-fold cover is written as $\tilde C_n^*(V)$: the
fibre of $\tilde C_n^*(V) \to C_n^*(V)$ at $A\in C_n^*(V)$ is just the
$n$-element subset $A$, and there is a map $\pi : \tilde C^*_n(V)
\to V-\{ 0\}$ given by the inclusion of $A$ in $V-\{ 0\}$.
\end{defn}
We shall think of a map $\phi  : S(V) \to C^*_n(V)$ from the unit sphere in
$V$ (for some Euclidean inner product) as an $n$-valued map
$$
S(V) \multimap V-\{ 0\}\, .
$$
We have an associated structured finite-valued map specified by 
the $n$-fold cover $p:\tilde X \to X =S(V)$ pulled back by
$\phi$ from $\tilde C^*_n(V) \to C^*_n(V)$ and the map $f:\tilde X \to
S(V)=X$ obtained by composing with $\pi : \tilde C^*_n(V) \to V-\{ 0\}$ and
$v\mapsto v/\| v\| : V-\{ 0\} \to S(V)$.
\begin{defn} \label{degree}
The {\it degree} of $\phi$ is the integer $\deg (\phi)\in\Zz$ constructed as
follows.
For $\dim V >1$, we
define $\deg (\phi )\in\Zz$ to be $\deg (f/p)$ (for either orientation of
$S(V)$).
For $\dim V=1$, we make the {\it ad hoc} definition
$$
\deg (\phi ) = 
\# \{ x\in \tilde X \st p(x)=f(x) \} - n \, ,
$$
so that in this case $-n\leq \deg (\phi )\leq n$.
\end{defn}
\begin{rem}
In general, the degree is related to the Lefschetz index $L(f/p)$, as described
in \cite{nval}, by the identity
$$
\deg (\phi ) = (-1)^{\dim V} (n- L(f/p)).
$$
\end{rem}

\begin{lem}
Suppose that a finite group $G$ of order $n$ acts freely on $S(V)$. Then
there is an inclusion $\iota : S(V)/G \into C_n^*(V)$ 
defined by regarding a $G$-orbit 
in $S(V)$ as an $n$-element subset of $V-\{ 0\}$.
Given a map $\psi : S(V)\to S(V)/G$, the degree of $\phi =\iota\circ\psi$ is
equal to

\smallskip

$\deg (\psi )$ if $\dim V=2\,$;

\smallskip

$n\cdot\deg (\psi )$ if $\dim V >2$ and the action of $G$ is 
orientation-preserving\,;

\smallskip

$0$ if $\dim V >2$ and $G$ does not preserve the orientation. 
\end{lem}
\begin{proof}
If $\dim V>2$, $S(V)$ is simply connected and $\psi$ lifts to
$\tilde \psi : S(V)\to S(V)$. We have 
$$
\deg (\phi ) =\sum_{g\in G}\deg (g\tilde\psi ) \, .
$$
But $\deg (g\tilde\psi )=\deg (\tilde\psi )$ if $\det (g)=+1$,
$-\deg (\tilde\psi )$ if $\det (g)=-1$.
If $G$ preserves the orientation of $V$, the manifold $S(V)/G$ is oriented,
so that $\deg (\psi )$ is defined and is equal to $\deg (\tilde\psi )$.
If $G$ does not preserve the orientation, then the terms cancel to zero.

If $\dim V=2$, then $G$ is cyclic and orientation-preserving. 
It is enough, by homotopy invariance, to look at the special case
$V=\Cc$, $G=\mu_n$, in which we identify $S(V)/G$ with $S(\Cc )$ by
$[w]\mapsto w^n$ and describe $\psi: S(\Cc )\to S(\Cc )$ by $\psi (z)=z^d$.
Then $\tilde X =\{ (z,w) \in S(\Cc )\times S(\Cc )\st z^d=w^n\}$
is a union of $k\geq 1$ circles, where $k$ is the highest common factor
of $d$ and $n$,
and $f: \tilde X \to X=S(\Cc )$ maps $(z,w)$ to $w$.
The degree is easily checked to be equal $d$, for example, by looking
at the $|d|$ points of $f^{-1}(1)$.
\end{proof}
\section{A Poincar\'e-Hopf theorem for $n$-valued sections of a vector bundle}
\begin{defn}
Given a real vector bundle $\xi$ over $M$, we write
$C^*_n(\xi )\to M$ for the locally trivial fibre bundle with fibre
at $x\in M$ the configuration space $C^*_n(\xi_x)$.
(For further material on fibrewise configuration spaces, see
\cite[(II.14)]{FHT}.)
A section $\sigma$ of $C^*_n(\xi )$ determines an $n$-element subset
of $\xi_x-\{ 0\}$ at each point $x$.
We shall call $\sigma$ an {\it $n$-valued nowhere-zero section} of
the vector bundle $\xi$.
\end{defn}
The main theorem is now a direct generalization and consequence of
the classical Poincar\'e-Hopf theorem.
\begin{thm} \label{main}
Suppose that $\xi$ is an $m$-dimensional real vector bundle, $m\geq 1$,
over a closed $m$-manifold $M$ and suppose that, for some $n\geq 1$,
$\sigma$ is an $n$-valued nowhere-zero section of $\xi$ over the complement
of a finite set $\Sing$.
\par\noindent {\rm (a).}
One can assign to each point of $\Sing$ a local index 
of $\sigma$ in $\Zz /2\Zz$,
and the sum of the local indices is equal to $n\cdot w_m(\xi )[M]\in\Zz /2\Zz$.
\par\noindent {\rm (b).}
Suppose that $w_1(\xi )=w_1(M)$ and an isomorphism is chosen
between $\delta$ and the orientation bundle of $\xi$.
Then one can assign to each point of $\Sing$ a local index 
of $\sigma$ in $\Zz$,
and the sum of the local indices is equal to 
$n\cdot e(\xi )[M]\in\Zz$.
\end{thm}
The proof follows (but is, perhaps, a little simpler than)
the argument used in \cite[Theorem 1.1]{CG}. 
\begin{proof}
The local index at a point $x\in \Sing$ is defined just as in 
Theorem \ref{poincare}.
Taking a small disc $D(\tau_xM)$ in $M$ centred at $x$ and a trivialization of the restriction
$\xi \, | \, D(\tau_x M)\iso D(\tau_x M)\times\xi_x$,
we obtain from $\sigma$ a map $\phi_x : S(\tau_x M) \to C^*_n(\xi_x)$.
The local index is the degree of $\phi_x$, reduced (mod $2$) in case (a).

We shall construct a closed manifold $\tilde M$, a 
continuous map $p: \tilde M \to M$
such that 
(i) the orientation bundle of $\tilde M$ is the pullback
$p^*\delta$ of the orientation bundle of $M$, 
(ii) $\tilde \Sing=p^{-1}(\Sing )$ is finite
and (iii) $p\, |\, : \tilde M -\tilde \Sing \to M-\Sing$ is an $n$-fold covering
space, and a nowhere-zero section $s$ of $p^*\xi \, |\, \tilde M-\tilde \Sing$ 
with the property that the local index of $\sigma$ at $x\in \Sing$ is equal
to the sum of the local indices of $s$ at the points of $p^{-1}(x)$.

The result then follows by applying Theorem \ref{poincare} to
$s$ and using the identity
$$
e(p^*\xi )[\tilde M] = p_!(p^*e(\xi ))[M] = n\cdot e(\xi )[M],
$$
depending on the fact that $p$ has degree (defined using the
isomorphism between the orientation bundle of $\tilde M$ and
$p^*\delta$) equal to $n$ (that is, $p_!(1)=n$).

To construct $\tilde M$, we choose disjoint discs
$D(\tau_xM)$ about the points $x\in \Sing$ and
let $W$ be the complement of the union of
open unit discs $B(\tau_x M)$ in $M$. 
It is a manifold with boundary the union of the spheres $S(\tau_x M)$.
The pullback by $\sigma$ of the $n$-fold cover $\tilde C^*_n(\xi )\to
C^*_n(\xi )$ is an $n$-fold cover $p : \tilde W \to W$.

Let us look at a particular $x\in \Sing$ and follow the notation
used in Definition \ref{degree} with $V=\tau_x M$. 
The map $p$ restricts over $X=S(V)$ to a cover $\tilde X \to X$ by
a union of spheres.
If $m\not=2$, this bundle is trivial, say $S(V)\times I \to S(V)$,
where $I$ is a finite set of cardinality $\# I=n$. 
(When $n=1$, there is a natural trivialization using
the order, given by a choice of orientation, in $\xi_x$.)
We close the singularity by gluing in the $n$ discs $D(V)\times I\to D(V)$.
If $m=2$, we can identify $V$ with $\Cc$ and describe $\tilde X \to X$
as $S(\Cc ) \times I \to S(\Cc )$, where $p(z,i)=z^{r_i}$, for positive
integers $r_i$ indexed by the finite set of components $I$.
The singularity is closed by gluing in $D(\Cc )\times I \to D(\Cc ):
(z,i)\mapsto z^{r_i}$.

Following this procedure for each $x\in \Sing$, we produce the required  manifold
$\tilde M$ and map $p:\tilde M\to M$ extending the finite cover
$\tilde W \to W$ to satisfy the conditions (i)-(iii).
The section $\sigma$ on $W$ produces a nowhere-zero
section $s$ of $p^*\xi$ on $\tilde W$,
and the identification of the local index of $\sigma$ at $x$ with the
sum of the local indices of $s$ at the $\# I$ points of $p^{-1}(x)$
is intrinsic to the Definition \ref{degree} of the degree.
\end{proof}

\begin{ex}
Suppose that $G$ is a finite group of order $n$ and that
$\xi$ is a $G$-vector bundle, over the trivial $G$-space $M$,
with $G$ acting freely on $S(\xi )$.
We have an embedding $S(\xi )/G \into C^*_n(\xi )$.
So a section of $S(\xi )/G$ on the complement of a finite set $\Sing$
determines an $n$-valued nowhere-zero section of $\xi$ outside $\Sing$.
This construction thus includes the examples: 

\smallskip

\quad $P(\eta )$ where $\eta$ is a real vector bundle,
from Section 2;

\smallskip

\quad $S(\eta )/\mu_n$ where $\eta$ is a complex vector bundle,
from Section 3;

\smallskip

\par\noindent
and the case $S(\eta )/G$ where $\eta$ is a quaternionic vector bundle
and $G$ is a subgroup of $\Sp (1)$, which does not seem to be tractable
by the methods of Sections 2 and 3.
\end{ex}
\begin{ex}
Suppose that $s_1,\ldots ,\, s_n$ are $n$ sections of $S(\xi )$.
Then we can construct a section $\sigma$
of $C^*_n(\xi )$ by choosing $n$ distinct
positive real numbers $t_i>0$ and setting $\sigma (x)=\{ t_1s_1(x),\,\ldots ,\,
t_ns_n(x)\}$.

This observation shows that, for $m>1$, the local index of an $n$-valued
section may take any integer value (and justifies the
concluding Remark 4.2 in \cite{CG}).
\end{ex}

\par\noindent {\bf Acknowledgment.}
I am grateful to R.~F.~Brown for introducing me to the concept of an $n$-valued map and for helpful comments on an earlier (2017)
version of this paper.
\end{document}